\documentclass[12pt]{article}

\usepackage{lipsum} 
\usepackage{hyperref} 
\usepackage{authblk} 
\usepackage{geometry} 
\usepackage{graphicx}
\usepackage{multirow}
\usepackage{amsmath,amssymb,amsfonts}
\usepackage{amsthm}
\usepackage{mathrsfs}
\usepackage[title]{appendix}
\usepackage{xcolor}
\usepackage{textcomp}
\usepackage{manyfoot}
\usepackage{booktabs}
\usepackage{algorithm}
\usepackage{algorithmicx}
\usepackage{algpseudocode}
\usepackage{listings}
\usepackage[numbers,sort&compress]{natbib}

\newcommand{\keywords}[1]{\par\noindent\textbf{Keywords:} #1}

\theoremstyle{plain}
\newtheorem{theorem}{Theorem}[section]

\newtheorem{lemma}[theorem]{Lemma}
\newtheorem{assumption}[theorem]{Assumption}

\theoremstyle{definition}

\newtheorem{remark}{Remark}[section]
\newtheorem{definition}{Definition}[section]

\title{Approximate Controllability of Fractional Evolution Equations with Nonlocal Conditions via operator theory}
\author[1]{Dev Prakash Jha}
\author[2]{Raju K George}
\affil[1,2]{\small\textit{Mathematics, Indian Institute of Space Science and Technology,}\\ \small\textit{Valiamala, Thiruvananthapuram 695547, Kerala, India}}

\date{\today}

\begin{document}

\maketitle

\begin{abstract}
This paper investigates the existence and uniqueness of mild solutions, as well as the approximate controllability, of a class of fractional evolution equations with nonlocal conditions in Hilbert spaces. Sufficient conditions for approximate controllability are established through a novel approach to the approximate solvability of semilinear operator equations. The methodology utilizes Green’s function and constructs a control function based on the Gramian controllability operator. The analysis is based on Schauder’s fixed point theorem and the theory of fractional order solution operators and resolvent operators. To demonstrate the feasibility of the proposed theoretical results, an illustrative example is provided.
\end{abstract}

\keywords{Semi-linear systems, Fractional calculus, Mild solution, Non-autonomous systems,  approximate controllability, Evolution operator, Schauder's fixed-point theorem, Ascoli-Arzelà theorem}

\section{Introduction}
\label{sec:intro}
Fractional differential equations have recently emerged as valuable tools for modeling various phenomena across numerous fields of science and engineering. These equations find applications in areas such as viscoelasticity, electrochemistry, control systems, porous media, and electromagnetics, among others (see \cite{metzler1995relaxation,mainardi1997fractional,hilfer2000applications,gaul1991damping,diethelm1999solution}). Significant advancements in the theory and applications of fractional differential equations have been made in recent years, as detailed in the monographs by Kilbas et al. \cite{kilbas2006theory}, Miller and Ross \cite{miller1993introduction}, Podlubny \cite{podlubny1998fractional}, as well as in the papers \cite{jha2024exact,jha2024existence,sakthivel2011approximate,liu2015approximate,zhou2009existence} and references therein.\\

Throughout this paper, unless otherwise specified, we use the following notations. We assume that $E$ is a Hilbert space equipped with the norm $\| \cdot \|$. Let $J = [0, a] \subset \mathbb{R}$, and let $C(J, E)$ denote the Banach space of continuous functions from $J$ to $E$ with the norm $\|u\| = \sup_{t \in J} \|u(t)\|$, where $u \in C(J, E)$. The goal of this paper is to establish sufficient conditions for the approximate controllability of certain classes of abstract fractional evolution equations with control of the form:
\begin{equation}\label{eq:P}\tag{P}
	\begin{cases}
		{}^{C}D^{\alpha} u(t) = Au(t) + Bv(t) + f(t, u(t)), & t \in J, \\
		u(0) = \sum_{k=1}^{m} c_k u(t_k),
	\end{cases}
\end{equation}
where the state variable $u(\cdot)$ takes values in the Hilbert space $E$; ${}^{C}D^\alpha$ is the Caputo fractional derivative of order $0 < \alpha < 1$; $A$ is the infinitesimal generator of a $C_0$-semigroup $T(t)$ of bounded operators on $E$; the control function $v(\cdot)$ belongs to $L_2(J, U)$, with $U$ being a Hilbert space; $B$ is a bounded linear operator from $U$ into $X:=L^2([0,a];E)$; $f : J \times E \to E$ is a given function satisfying certain assumptions; and $ u(t_k)$ are  elements of $E$ for any $k=1,2,3,...,m.$\\

On the other hand, the nonlocal initial condition proves more effective than the classical initial condition \( u(0) = u_0 \) in certain physical contexts. For instance, in 1993, Deng \cite{deng1993exponential} applied the nonlocal condition (1.2) to model the diffusion of a small amount of gas in a transparent tube. In this scenario, condition (1.2) allows for additional measurements at times \( t_k \), where \( k = 1, 2, \dots, m \), providing more precise data compared to relying solely on the measurement at \( t = 0 \). Furthermore, in 1999, Byszewski \cite{byszewski1999existence} highlighted that if \( c_k \neq 0 \) for \( k = 1, 2, \dots, m \), the results could be applied in kinematics to track the evolution of a physical object's position \( t \rightarrow u(t) \), even when the positions \( u(0), u(t_1), \dots, u(t_m) \) are unknown, as long as the nonlocal condition (1.2) holds.\\
Therefore, in describing certain physical phenomena, the nonlocal condition may be more applicable than the standard initial condition \( u(0) = u_0 \). The significance of nonlocal conditions has also been explored in \cite{jha2024exact,ezzinbi2007existence,liang2015controllability,wang2017approximate,byszewski1991theorems,boucherif2009semilinear}.

This article is structured as follows: In \autoref{sec:prel},  we recall some fundamental definitions and results on the Caputo fractional derivatives. \autoref{sec:Redu}, introduces relevant notations, concepts, hypotheses, and critical results about the reduction of the controllability problem into a solvability problem. In \autoref{sec:Mild}, we investigate the existence and uniqueness of mild solutions of semi-linear systems (\ref{eq:P}). Finally,\autoref{sec:Approx} we explore the approximate controllability of systems \autoref{sec:Appli}, presenting an example to demonstrate the application of the results obtained.

\section{Preliminaries}
\label{sec:prel}
In this section, we introduce some notations, definition, and preliminary facts which are used throughout this paper. Let E and U be two Hilbert spaces with the norms $\lVert \cdot \rVert$ and $\lVert \cdot \rVert_{U}$ , respectively. We denote by $C(J,E)$ the Hilbert space of all continuous functions from interval $J$ into $E$ equipped with the supremum norm\\
\begin{equation}\label{eq:pre_1}
	\lVert u\rVert_C= \sup_{t \in J} \lVert u(t)\rVert, \quad u \in C(J,E),
\end{equation}
 
Let $\mathcal{L}(E)$ be the Banach space of all linear and bounded operators in $E$ endowed with the topology defined by the operator norm. Let $L^2(J,U)$ be the Banach space of all $U$-value Bochner square integrable functions defined on $J$ with the norm
\begin{equation}\label{eq:pre_2}
	\|u\|_2 = \left( \int_0^a \|u(t)\|_U^2 dt \right)^{\frac{1}{2}}, \quad u \in L^2(J,U).
\end{equation}

\begin{definition}\label{def:preli_1}
	(\cite{kilbas2006theory})  The fractional integral of order $\alpha$ with the lower limit $0$ for a function $f$ is defined as
\[
I^\alpha f(t) = \frac{1}{\Gamma(\alpha)} \int_0^t \frac{f(s)}{(t-s)^{1-\alpha}} \, ds, \quad t > 0, \, \alpha > 0,
\]
provided that the right-hand side is pointwise defined on $[0, \infty)$, where $\Gamma$ denotes the gamma function.
\end{definition}

\begin{definition}(\cite{kilbas2006theory})\label{def:preli_2}
	 The Caputo fractional derivative of order $\alpha > 0$ with the lower limit 0 for a function $f$ is defined as
	\[
	{}^C D_t^\alpha f(t) = \frac{1}{\Gamma(1-\alpha)} \int_0^t (t-s)^{-\alpha} f'(s) ds, \quad t > 0,
	\]
	where $f(t)$ is absolutely continuous.
\end{definition}

\begin{definition}\label{Def:2}
     The Riemann--Liouville derivative of order $\alpha$ with the lower limit $0$ for a function $f : [0, \infty) \to \mathbb{R}$ is given by
\[
{}^{L}D^\alpha f(t) = \frac{1}{\Gamma(n-\alpha)} \frac{d^n}{dt^n} \int_0^t \frac{f(s)}{(t-s)^{\alpha-n+1}} \, ds, \quad t > 0, \, n-1 < \alpha < n.
\]
\end{definition}

\begin{definition}\label{Def:3}
     The Caputo derivative of order $\alpha$ for a function $f : [0, \infty) \to \mathbb{R}$ is defined as
\[
{}^{C}D^\alpha f(t) = {}^{L}D^\alpha \left(f(t) - \sum_{k=0}^{n-1} \frac{t^k}{k!} f^{(k)}(0)\right), \quad t > 0, \, n-1 < \alpha < n.
\]
\end{definition}

 \begin{remark}
    \begin{enumerate}
    \item If $f \in C^n[0, \infty)$, then
    \[
    {}^{C}D^\alpha f(t) = \frac{1}{\Gamma(n-\alpha)} \int_0^t \frac{f^{(n)}(s)}{(t-s)^{\alpha-n+1}} \, ds = I^{n-\alpha} f^{(n)}(t), \quad t > 0, \, n-1 < \alpha < n.
    \]
    \item The Caputo derivative of a constant function is equal to zero.
    \item if $f $ is an abstract function with values in $X,$ then the integrals in Definitions (\ref{Def:2}) and (\ref{Def:3}) are taken in the Boucher's sense.
\end{enumerate}
\end{remark}

If $f$ is an abstract function with values in $E$, then the integrals which appear in Definition (\ref{def:preli_1}) and Definition (\ref{def:preli_2}) are taken in Bochner’s sense, that is: a measurable function $f$ maps from $[0, +\infty)$ to $E$ is Bochner integrable if $\|f\|$ is Lebesgue integrable.
\begin{definition}\label{def:preli_3}
	(\cite[Definition 2.3]{bazhlekova2001fractional}) A function $T_\alpha : \mathbb{R}^+ \to \mathcal{L}(E)$ is called an $\alpha$-order solution operator generated by $A$ if the following conditions are satisfied:
	\begin{itemize}
		\item[(i)] $T_\alpha(t)$ is strongly continuous for $t \geq 0,$ and $T_\alpha(0) = I$;
		\item[(ii)] $T_\alpha(t)D(A) \subseteq D(A)$ and $AT_\alpha(t)x = T_\alpha(t)Ax$ for all $x \in E$ and $t \geq 0$;
		\item[(iii)] For all $x \in D(A)$ and $t \geq 0$, $T_\alpha(t)x$ is a solution of the following operator equation
		\[
		u(t) = x + \frac{1}{\Gamma(\alpha)} \int_0^t (t-s)^{\alpha-1} Au(s) ds.
		\]
	\end{itemize}
\end{definition}
\begin{definition}\label{def:preli_4}
	(\cite[Definition 2.13]{bazhlekova2001fractional}) Let $0 < \theta_0 \leq \frac{\pi}{2}$ and $\omega_0 \in \mathbb{R}$. An $\alpha$-order solution operator $\{T_\alpha(t)\}_{t \geq 0}$ is called analytic if it admits an analytic extension to a sector $\Sigma_{\theta_0} := \{z \in \mathbb{C} \setminus \{0\} : |\arg z| < \theta_0\}$ and the analytic extension is strongly continuous on $\Sigma_\theta$ for every $\theta \in (0, \theta_0)$. An analytic $\alpha$-order solution operator $T_\alpha(z)$ ($z \in \Sigma_\theta$), generated by $A$, is said to be of analyticity type $(\omega_0, \theta_0)$ if for each $\theta \in (0, \theta_0)$ and $\omega > \omega_0$, there exists a positive constant $M = M(\omega, \theta)$ such that
\[
\|T_\alpha(z)\| \leq M e^{\omega \Re z}, \quad z \in \Sigma_\theta.
\]
If $A$ generates an analytic $\alpha$-order solution operator $T_\alpha$ of analyticity type $(\omega_0, \theta_0)$, then we write $A \in \mathcal{A}^\alpha(\omega_0, \theta_0)$.
\end{definition}

\begin{definition}(\cite[Definition 2.3]{araya2008almost})\label{def:preli_5}
	Let $A$ be a closed and linear operator with domain $D(A)$ defined on $E$ and $\alpha > 0$. Let $\rho(A)$ be the resolvent set of $A$. We call $A$ the generator of an $\alpha$-order resolvent operator if there exist $\omega \geq 0$ and a strongly continuous function $S_{\alpha} : \mathbb{R}^{+} \rightarrow \mathcal{L}(E)$ such that $	\{ \nu \in \mathbb{C} : \text{Re} \, \nu > \omega \} \subset \rho(A)$ and 
	\[
  (\nu^{\alpha}I - A)^{-1}u = \int_{0}^{\infty} e^{-\nu t} S_{\alpha}(t)u \, dt, \quad \text{Re} \, \nu > \omega, \quad u \in E.
	\]
	In this case, $\{S_{\alpha}(t)\}_{t \geq 0}$ is called the $\alpha$-order resolvent operator generated by $A$.
\end{definition}

\begin{definition}\label{def:preli_6}
	An $\alpha$-order solution operator $\{T_{\alpha}(t)\}_{t \geq 0}$ is called compact if for every $t > 0$, $T_{\alpha}(t)$ is a compact operator. An $\alpha$-order resolvent operator $\{S_{\alpha}(t)\}_{t \geq 0}$ is called compact if for every $t > 0$, $S_{\alpha}(t)$ is a compact operator.
\end{definition}

By using the similar method which used in  (\cite[Lemma 10]{fan2014characterization}), we can get the continuity of the $\alpha$-order solution operator $\{T_{\alpha}(t)\}_{t \geq 0}$ and $\alpha$-order resolvent operator $\{S_{\alpha}(t)\}_{t \geq 0}$ in the uniform operator topology.

\begin{lemma}
	Suppose that $A \in \mathcal{A}^{\alpha}(\omega_{0}, \theta_{0})$, the $\alpha$-order solution operator $\{T_{\alpha}(t)\}_{t \geq 0}$ and $\alpha$-order resolvent operator $\{S_{\alpha}(t)\}_{t \geq 0}$ are compact. Then the following properties are valid for every $t > 0$:
	\begin{enumerate}
		\item $\lim_{h \rightarrow 0} \|T_{\alpha}(t + h) - T_{\alpha}(t)\| = 0$, $\lim_{h \rightarrow 0} \|S_{\alpha}(t + h) - S_{\alpha}(t)\| = 0$;
		\item $\lim_{h \rightarrow 0^{+}} \|T_{\alpha}(t + h) - T_{\alpha}(h)T_{\alpha}(t)\| = 0$, $\lim_{h \rightarrow 0^{+}} \|S_{\alpha}(t + h) - S_{\alpha}(h)S_{\alpha}(t)\| = 0$;
		\item $\lim_{h \rightarrow 0^{+}} \|T_{\alpha}(t) - T_{\alpha}(h)T_{\alpha}(t - h)\| = 0$, $\lim_{h \rightarrow 0^{+}} \|S_{\alpha}(t) - S_{\alpha}(h)S_{\alpha}(t - h)\| = 0$.
	\end{enumerate}
\end{lemma}

\begin{lemma}
	\label{lemma:2.2} (\cite[2.26]{bazhlekova2001fractional}\cite[Theorem 2.3]{shi2016study})
 If $A \in \mathcal{A}^{\alpha}(\omega_0, \theta_0)$, then for every $t > 0$ and $\omega > \omega_0$, we have
	\begin{equation*}
		\|T_{\alpha}(t)\|_{\mathcal{L}(E)} \leq Me^{\omega t} \quad \text{and} \quad \|S_{\alpha}(t)\|_{\mathcal{L}(E)} \leq Ce^{\omega t}(1 + t^{1-\alpha}). \tag{2.3}
	\end{equation*}
	
	Furthermore, let
	\begin{equation*}
		M_T := \sup_{t \in \mathbb{R}^+} \|T_{\alpha}(t)\|_{\mathcal{L}(E)}, \quad M_S := \sup_{t \in \mathbb{R}^+} Ce^{\omega t}(1 + t^{1-\alpha}). \tag{2.4}
	\end{equation*}
	
	We get that
	\begin{equation*}
		\|T_{\alpha}(t)\|_{\mathcal{L}(E)} \leq M_T, \quad \|S_{\alpha}(t)\|_{\mathcal{L}(E)} \leq e^{\omega t}M_S. \tag{2.5}
	\end{equation*}
	
	For more definitions and properties about $\alpha$-order solution operator and $\alpha$-order resolvent operator, please see the thesis \cite{bazhlekova2001fractional} and the papers \cite{li2012cauchy} \cite{lian2017approximate}.
\end{lemma}

\label{sec:preliminaries}

\section{Reduction of the controllability problem into a solvability problem}
\label{sec:Redu}
First, we introduce some notation. For any bounded linear operator $T$ , \( D(T) \), \( R(T) \), and \( N_0(T) \) 
denote the domain, range, and null space of \( T\), respectively. \( T^* \) denotes the adjoint 
of \( T \). \( \overline{F} \) and \( F^\perp \) respectively are the closure and orthogonal complement of a set \( F \).\\

Let $E$ and $H$ be Hilbert spaces as defined in the section \ref{sec:intro} . Let \( L: E \to H \) be a bounded linear operator, and \( N: E \to H \) a nonlinear operator. Consider the semilinear equation
\[
Lu + Nu = v,
\]
and the corresponding linear equation
\[
Lu = v,
\]
where \( v \in H \) is arbitrary.

Hereafter, we denote the semilinear and linear systems given above by \([E, H, L + N]\) and \([E, H, L]\), respectively.
\begin{definition}\label{def:Redu_1}
     \([E, H, L + N]\) (\([E, H, L]\)) is exactly solvable if the range of \( L + N \) (range of \( L \)) is \( H \). If the range is dense in \( H \), then the system is called approximately solvable.
\end{definition}

The theorem given below answers the following basic question.

If \([E, H, L]\)  is approximately solvable, when is \([E, H, L + N]\) approximately solvable?
\begin{theorem}\cite{joshi1990approximate}\label{theorem:Solvable}
    The semilinear system \([E, H, L + N]\) is approximately solvable under the following conditions on \( L \) and \( N \):
\begin{itemize}
    \item[(A)] \([E, H, L]\) is approximately solvable. 
    \item[(B)] \( N \) is compact and uniformly bounded, i.e., \( \|Nu\| \leq M, \, \forall u \in E, M \geq 0 \).
    \item[(C)] \( R(N) \subseteq R(L) \).
\end{itemize}
\end{theorem}
Consider the following linear system fractional differential system
\begin{equation}\label{eq:Q}\tag{Q}
	\begin{cases}
		{}^{C}D^{\alpha} u(t) = Au(t) + Bv(t), & t \in J, \\
		u(0) = \sum_{k=1}^{m} c_k u(t_k),
	\end{cases}
\end{equation}

  Define an operator $S: E \to H$ by\\

$$ Sx= \int_{0}^{t_1}G(t_1,\tau) s(\tau) d\tau$$

Let Y be the orthogonal complement of $N(S),$ that is $E=N(S) \oplus Y.$
We assume the following sufficient condition for the approximate controllability of the linear system (\ref{eq:Q})

\begin{assumption}
\begin{itemize}
	\item [(B1)]:  For each $x \in E$ there exists a $y \in \overline{R(B)}$ such that $Sx=Sy$.
\end{itemize}

\end{assumption}

please follow similar proof of the following article \cite{george1995approximate}

\begin{lemma}
\begin{itemize}
	\item[(i)] $R(S) = H$;
	\item[(ii)] $[\text{B1}] \, \iff  E = R(B) + N(S)$;
	\item[(iii)] $[\text{B1}] \, \implies \, R(SB) = H$;
	\item[(iv)] $[\text{B1}] \,  \implies \, (u_m + N(S)) \cap \overline{R(B)} \, \neq \, G  \, \text{ for each } u_m \in Y$.
\end{itemize}
\end{lemma}
\section{Existence and uniqueness of mild solutions}
\label{sec:Mild}
Throughout this paper, we assume that
\[
\text{(H1)} \quad \sum_{k=1}^m |c_k| < \frac{1}{M_T}.
\]
now we reformulate assumption (H1) of the following:

\begin{equation}\label{eq:chen_1}
  \left\| \sum_{k=1}^m c_k T_\alpha(t_k) \right\| \leq M_T \sum_{k=1}^m |c_k| < 1.  
\end{equation}

By (\ref{eq:chen_1}) and the operator spectrum theorem, we know that
\begin{equation}\label{eq:chen_2}
    \mathcal{O} := \left( I - \sum_{k=1}^m c_k T_\alpha(t_k) \right)^{-1}
\end{equation}

exists, is bounded, and \( D(\mathcal{O}) = E \). Furthermore, by the Neumann expression, \( \mathcal{O} \) can be expressed by
\begin{equation}\label{eq:chen_3}
    \mathcal{O} = \sum_{n=0}^\infty \left( \sum_{k=1}^m c_k T_\alpha(t_k) \right)^n.
\end{equation}

Therefore,
\begin{equation}\label{eq:chen_4}
    \|\mathcal{O}\| \leq \sum_{n=0}^\infty \left\| \sum_{k=1}^m c_k T_\alpha(t_k) \right\|^n \leq \frac{1}{1 - M_T \sum_{k=1}^m |c_k|}.
\end{equation}

By the above discussion, \cite[Proposition 1.2]{pruss2012evolutionary} and \cite{shi2016study}, we know that the mild solution of the fractional evolution equation (\ref{eq:P}) with initial value \( u(0) \) can be expressed by
\begin{equation}\label{eq:chen_5}
    u(t) = T_\alpha(t) u(0) + \int_0^t S_\alpha(t - s) \big[ B v(s) + f(s, u(s)) \big] ds, \quad t \in J.
\end{equation}
From (\ref{eq:chen_5}) one gets that 
\begin{equation}\label{eq:chen_5_1}
    u(t_k) = T_\alpha(t_k) u(0) + \int_0^{t_k} S_\alpha(t_k - s) \big[ B v(s) + f(s, u(s)) \big] ds, \quad k=1,2,...,m.
\end{equation}
By (\ref{eq:P}) and (\ref{eq:chen_5_1}) we know that
\begin{equation}\label{eq:chen_6}
u(0) = \sum_{k=1}^m c_k T_\alpha(t_k) u(0) + \sum_{k=1}^m c_k \int_0^{t_k} S_\alpha(t_k - s) [Bv(s) + f(s, u(s))] \, ds. 
\end{equation}

By assumption (H1) and the definition of operator $\mathcal{O}$, we have
\begin{equation}\label{eq:chen_7}
u(0) = \sum_{k=1}^m c_k \mathcal{O} \int_0^{t_k} S_\alpha(t_k - s) [Bv(s) + f(s, u(s))] \, ds. 
\end{equation}

Therefore, (\ref{eq:chen_5}) and (\ref{eq:chen_7}) mean that

\begin{align}\label{eq:chen_8}
u(t) &= \sum_{k=1}^m c_k T_\alpha(t) \mathcal{O} \int_0^{t_k} S_\alpha(t_k - s) [Bv(s) + f(s, u(s))] \, ds \nonumber \\
&\quad + \int_0^t S_\alpha(t - s) [Bv(s) + f(s, u(s))] \, ds, \quad t \in J.
\end{align}

For convenience, we introduce the Green's function $G(t, s)$ as follows:
\begin{equation}\label{eq:chen_9}
G(t, s) = \sum_{k=1}^m \chi_{t_k}(s) T_\alpha(t) \mathcal{O} S_\alpha(t_k - s) + \chi_t(s) S_\alpha(t - s), \quad t, s \in [0, a], 
\end{equation}
where
\begin{equation}\label{eq:chen_10}
\chi_{t_k}(s) =
\begin{cases}
c_k, & s \in [0, t_k), \\
0, & s \in [t_k, a],
\end{cases}
\quad
\chi_t(s) =
\begin{cases}
1, & s \in [0, t), \\
0, & s \in [t, a].
\end{cases} 
\end{equation}

Thus, by (\ref{eq:chen_8}), (\ref{eq:chen_9}), and (\ref{eq:chen_10}), we know that the mild solution of FEE (\ref{eq:P}) can also be expressed by
\begin{equation}\label{eq:chen_11}
u(t) = \int_0^a G(t, s) [Bv(s) + f(s, u(s))] \, ds, \quad t \in J. 
\end{equation}

Therefore, we have the following definition:
\begin{definition}\label{def:Mild_1}
 A function $u \in C(J, E)$ is said to be a mild solution of FEE (\ref{eq:P}), if for any $v \in L^2(J, U)$, $u(t)$ satisfies the integral equation
\end{definition}

\begin{equation}\label{eq:chen_12}
u(t) = \int_0^a G(t, s) [Bv(s) + f(s, u(s))] \, ds, \quad t \in J,
\end{equation}
where $G(t, s)$ is the Green's function defined by (\ref{eq:chen_9}).

\begin{remark}
    From the above discussion, it is clear that the green function is a finite linear combination of compact operators on a finite interval, so the green function  $G(t,s)$ is compact for $t>s$ on on J. So it is bunded by $N>0$, that is $\lVert G(t,s) \rVert \leq N ,$ for $t>s.$
\end{remark}

\begin{itemize}
	\item[(H2)] The function $f : J \times E \to E$ is continuous and there exists a function $\varphi \in L^\beta(J, \mathbb{R}^+)$ with $0 \leq \beta < \alpha$ such that
    $$\|f(t, u)\| \leq L \lVert u\rVert_E + \varphi(t), \quad \forall u \in E, t \in J,$$ 
	where, $L>0,$ and the function $\varphi(\cdot)$ is such that $\bigg(\int_0^t \varphi^{\beta} (\tau) d\tau\bigg)^{\frac{1}{\beta}}=b<\infty.$
	\item[(H3)] There exists a function $\psi \in L^\gamma(J, \mathbb{R}^+)$ with $0 \leq \gamma < \alpha$ such that $\|Bv(t)\| \leq \psi(t)$ for all $t \in J$ and $v \in L^2(J, U)$.
    \item [(H4)] There exists $\mu > 0$ such that $\langle -Au, u \rangle_E \geq \mu \| u \|_E^2$ for all $u \in D(A)$.
     \item [(H5)] $\langle f(t,u_1) - f(t,u_2), u_1 - u_2 \rangle_E \leq 0$ for all $u_1, u_2 \in E$ and $t \in J$.
\end{itemize}

The following theorem is proofed with the help of \cite{chen2020existence} we will prove.
\begin{theorem}\label{theorem:Th_1_Mild}
     Let $A \in \mathcal{A}^\alpha(\omega_0, \theta_0)$ with $\theta_0 \in \left(0, \frac{\pi}{2}\right)$ and $\omega_0 \in \mathbb{R}$. Assume that the $\alpha$-order solution operator $\{T_\alpha(t)\}_{t \geq 0}$ and $\alpha$-order resolvent operator $\{S_\alpha(t)\}_{t \geq 0}$ are compact. Suppose also that the assumptions (H1),(H2) and (H3) are satisfied, then FEE (\ref{eq:P}) has at Unique mild solution on $J$.
\end{theorem}
\begin{proof}
    I have a proof.
\end{proof}

\section{Approximate}
\label{sec:Approx}
Let $X = L^2[J, E]$, and $\overline{\text{R}(B)}$ denotes the closure of the range of $B$. 
Let $W: \overline{\text{R}(B)} \to X$ be an operator (called solution operator) defined by $W\mu = y$ 
where $y$ is the unique mild solution of the equation

\begin{equation}\label{eq:R}\tag{R}
	\begin{cases}
		{}^{C}D^{\alpha} y(t) = Ay(t) + \mu(t) + f(t, y(t)), & t \in J, \\
		y(0) = \sum_{k=1}^{m} c_k u(t_k),
	\end{cases}
\end{equation}

In this section, we obtain different sets of sufficient conditions to guarantee the existence of 
the solution ooperator,$W$ and we study its behaviour under various situations.\\

Let us define the operators $K, N: L^2(J,E) \to L^2(J,E)$ as follows
\begin{equation}\label{eq:Raju_2}
(K\mu)(t) = \int_{0}^{t} G(t,s)\mu(s) \, ds \quad \text{and} \quad (N\mu)(t) = f(t, \mu(t)).
\end{equation}

Let $F$ is the Nemytskii operator from $X$ into itself defined as $[Fz](t)=f(t,z(t))$ , then $(N\mu)(t)=[F\mu](t).$
\begin{definition}[Reachable set \( K_a(F) \):]\label{def:Mild_2}
The reachable set of the system (\ref{eq:P}) is defined as 
\begin{multline*}
\mathcal{K}_a(f)= \{ y_\mu(a) : \mu \in l^2(J,U) \text{ and } y_\mu \in X \text{ denotes the mild solution of (\ref{eq:P})} \\
\text{corresponding to the control } \mu \}.
\end{multline*}
\end{definition}

\begin{definition}[Approximate controllability]

 The system (\ref{eq:P}) is said to be approximately controllable if $\mathcal{K}_a(f)$ is dense in $E$.

\end{definition}

$\mathcal{K}_a(0)$ denotes the reachable set of the linear system corresponding to (\ref{eq:P}), which is obtained by taking $f = 0$. It is then clear that the linear system is approximately controllable if $\mathcal{K}_a(0)$ is dense in $E$.

\begin{lemma}\label{lemma:Lemma_2.1}
Suppose that $A$ and $f(t,u)$ satisfy assumptions [H4] and [H5]. Then the solution operator $W$ is well-defined and continuous. Moreover, it satisfies a growth condition
\begin{equation}\label{eq:Raju_1}
\| Wu \|_E < \left( \frac{a}{\mu} + 1 \right) k \| v \|_E + \frac{bk}{\mu}.
\end{equation}
\end{lemma}
 \begin{proof}

Obviously, $K$ is a bounded linear operator with $\| K \| \leq k$ and $N$ is a continuous and bounded nonlinear operator (called Nymitskii operator). By using these notations, the solvability of the nonlinear integral equation (\ref{eq:chen_12}) is equivalent to the solvability of the following operator equation
\begin{equation}\label{eq:Raju_3}
u = Kv + KNu.
\end{equation}

Now, by the hypothesis of the lemma, $N$ satisfies the following
\begin{equation}\label{eq:Raju_4}
\langle Nu_1 - Nu_2, u_1 - u_2 \rangle_H \leq 0 \quad \text{for all } u_1, u_2 \in H.
\end{equation}

Now we claim that \( K \) satisfies
\[
\langle Ku, u \rangle_H \geq \mu \| Ku \|_H^2 \quad \text{for all } u \in H.
\]
For that, first we prove that the above inequality is true for all \( u \in L^2(J, D(A)) \). For any \( u \in L^2(J, D(A)) \), define \( f(t) = \int_{0}^{t} G(t,s)u(s) \, ds \). Since \( G(t,s) \) is a strongly continuous evolution operator (XXXXXXXX), we have that \( f(t) \in D(A) \) and

\[
f'(t) = u(t) + A \int_{0}^{t} G(t,s)u(s) \, ds
\]

\[
\langle Ku, u \rangle_H = \int_{0}^{t} \left\langle f(t), f'(t) - A \int_{0}^{t} G(t,s)u(s) \, ds \right\rangle_E \, dt
\]

\[
= \int_{0}^{t} \langle f(t), f'(t) \rangle_E \, dt + \int_{0}^{t} \langle f(t), -Af(t) \rangle_E \, dt.
\]

However,

\[
\int_{0}^{t} \langle f(t), f'(t) \rangle_E \, dt = \langle f(t), f(t) \rangle_E \Big|_{t_0}^{t} - \int_{0}^{t} \langle f'(t), f(t) \rangle_E \, dt = \frac{1}{2} \| f(t) \|_E^2 \geq 0.
\]

Therefore, by using assumption \([A1]\) we have

\[
\langle Ku, u \rangle_H \geq \mu \int_{0}^{t} \langle f(t), f(t) \rangle_E \, dt = \mu \| Ku \|_H^2 \quad \text{for all } u \in L^2(J, E).
\]

Now for any \( u \in H \), consider the sequence \( \{ u_n \} \) in \( L^2(J, D(A)) \) such that \( u_n \to u \) in \( H \). By taking limit, one can obtain the required estimate of \( \langle Ku, u \rangle \) for all \( u \in H \).

The existence of a unique solution \( u \) for the equation (\ref{eq:Raju_3}) and the growth condition (\ref{eq:Raju_1}) follow along the same line as in theorem 5.1 of Joshi and George \cite{joshi1989controllability}.

\end{proof}
The assumption \([H4]\) made on \( A \) can be weakened by imposing strong monotonicity on \( f(t,u) \) as we see in the following lemma. In order to get compactness on \( W \) we assume that \( G(t,s) \) is compact for \( t - s > 0 \).

\begin{assumption}
    \begin{itemize}
	\item[(H6):] \( \langle -Au, u \rangle_{D(A)} \geq 0 \) for all \( u \in D(A) \).
	\item[(H7):] There exists \( \beta > 0 \) such that
	\[
	\langle f(t,u_1) - f(t,u_2), u_1 - u_2 \rangle_E \leq -\beta \| u_1 - u_2 \|_E^2 \quad \forall \, u_1,u_2 \in E, \, t \in J.
	\]
\end{itemize}
\end{assumption}

\begin{lemma}
 Suppose that \( A \) and \( f(t,u) \) satisfy assumptions \([H6], [H2]\) and \([H7]\). Then the solution operator \( W \) is well-defined, continuous and compact. Further, it satisfies a growth condition
\begin{equation}\label{eq:Raju_5}
	\| W v \|_H < c_0 + c \| v \|_H,
\end{equation}

where
\[
c_0 = m \sqrt{(t_1 - t_0) \|u_0\|_H + \sqrt{(t_1 - t_0)b} e^{m \alpha (t_1 - t_0)}}
\]
and
\[
c = m (t_1 - t_0) e^{m \alpha (t_1 - t_0)}.
\]
\end{lemma}

\begin{proof}
First we prove the growth condition of the solution operator. The solution $u(t)$ satisfies
\[
\|u(t)\|_E \leq  m \int_{0}^t \|f(s, u(s))\|_E \, ds + m \int_{0}^t \|v(s)\|_E \, ds
\]
\[
\leq  m \int_{0}^t b(s) \, ds + ma \int_{0}^t \|u(s)\|_E \, ds + m \int_{0}^t \|v(s)\|_E \, ds
\]
\[
\leq \sqrt{(t_1 - t_0)b} + \sqrt{(t_1 - t_0)\lVert v \rVert_E} + ma \int_{0}^t \|u(s)\|_E \, ds.
\]

Using Gronwall's inequality, we get
\[
\|u(t)\|_E \leq \sqrt{(t_1 - t_0)b} e^{m a (t_1 - t_0)} + m \sqrt{(t_1 - t_0)} e^{m a (t_1 - t_0)} \|v\|_X.
\]

Thus we have
$$\left( \int_{0}^{a} \|u(t)\|_E^2 \, dt \right)^{1/2} $$
\begin{align*}
	&\leq \left( \int_{0}^{t_1} \left( \sqrt{(t_1 - t_0)b} e^{m \alpha (t_1 - t_0)} + m \sqrt{(t_1 - t_0)} e^{m \alpha (t_1 - t_0)} \|v\|_H \right)^2 \, dt \right)^{1/2} \\
	&\leq \left( \int_{0}^{t_1} \left( m \sqrt{(t_1 - t_0)b} e^{m \alpha (t_1 - t_0)} \right)^2 \, dt \right)^{1/2} \\
	&\quad + \left( \int_{0}^{t_1} m^2 (t_1 - t_0) e^{2m \alpha (t_1 - t_0)} \|v\|^2_{L^2} \, dt \right)^{1/2}.
\end{align*}

That is,
\[
\|Wx\|_H \leq m \sqrt{ \sqrt{(t_1 - t_0)b} e^{m \alpha (t_1 - t_0)}} + m (t_1 - t_0) e^{m \alpha (t_1 - t_0)} \|v\|_H.
\]

Using the hypotheses [H6] and [H5], it can be shown as in Lemma \ref{lemma:Lemma_2.1} that the operators $K$ and $N$ defined by (\ref{eq:Raju_2}) satisfy the following conditions:
\[
\langle Ku, u \rangle_H \geq 0 \quad \text{for all } u \in H,
\]
\[
\langle Nu_1 - Nu_2, u_1 - u_2 \rangle_H = -\beta \|u_1 - u_2\|^2_H \quad \text{for all } u_1, u_2 \in H.
\]

Also,  the operator $K$ is compact.
As in Lemma \ref{lemma:Lemma_2.1} we are looking for the solvability of (\ref{eq:Raju_3}). To prove the uniqueness, let $u_1$ and $u_2$ be two solutions of (\ref{eq:Raju_3}) for a given $v$. Thus,
\[
u_1 - u_2 = KNu_1 - KNu_2
\]
and, hence, we have
\[
\langle u_1- u_2, Nu_1 - Nu_2 \rangle_H = \langle KNu_1 - KNu_2, Nu_1 - Nu_2 \rangle_H.
\]
However, from [H7] we have that the LHS is less than or equal to $-\beta \|u_1 - u_2\|_H^2$ and [H6] implies that the RHS is nonnegative. Thus it can happen only when $u_1 = u_2$, proving the uniqueness of the solution.

To show the existence of a solution for (\ref{eq:Raju_3}), let us define a sequence of operators $\{K_n\}$ defined by $K_n = K + (1/n)I_H$, $I_H$ being the identity operator on $H$, $n = 1, 2, 3, \ldots$. Thus for each $n$, $\langle K_n u, u \rangle_H \geq (1/n)\|u\|_H^2$ for all $u \in H$. Now the estimate $\|K_n u\| \leq (k + (1/n))\|u\|$ implies that
\[
\langle K_n u, u \rangle_H \geq n/(1 + kn)^2 \|K_n u\|^2 \quad \text{for all } u \in H
\]
and, hence, by the main result in Hess \cite{hess1971nonlinear} we have that $I - K_n N$ is invertible. Therefore, for every $v$, there exists a solution $u_n$ for the equation
\[
u_n = K_n N u_n  + Kv. 
\]

We now claim that this solution $u_n$ is bounded for large $n$. For that, let us write $u_n$ in the following form
\[
u_n(t) =  \int_{0}^t G(t, s)u(s) \, ds + \int_{0}^t G(t, s)f(s, u_n(s)) \, ds +\frac{f(t, u_n(t))}{n} \, 
\]
\[
\|u(t)\|_E \leq  m \int_{0}^t \|f(s, u_n(s))\|_E \, ds + m \int_{0}^t \|v(s)\|_E \, ds + \frac{\lVert f(t, u_n(t))\rVert}{n} 
\]
\[
\leq  \sqrt{(t_1 - t_0)b} + \sqrt{(t_1 - t_0)b} + ma \int_{0}^t \|v(s)\|_E \, ds + \frac{b(t) + a\|u_n(t)}{n}
\]

\[
(1 - a/n)\|u(t)\| \leq  \sqrt{(t_1 - t_0)b} + b/n + \sqrt{(t_1 - t_0)} \|u\|_H + ma \int_{0}^t \|u_n(s)\| \, ds.
\]
\[
\|u(t)\| \leq m(n) (b(n) + \sqrt{(t_1 - t_0)} \|u\|_H) + m(n)a \int_{0}^t \|u_n(s)\| \, ds,
\]
where $m(n) = mn/(n - a)$ and $b(n) = \sqrt{(t_1 - t_0)b} + b/n$.

Using Gronwall's inequality, we get
\[
\|u_n(t)\|_E \leq m(n) b(n) e^{m \alpha (t_1 - t_0)} + m(n)\sqrt{(t_1 - t_0)}e^{m \alpha (t_1 - t_0)}\|v\|_E.
\]

Thus we have
\[
\left( \int_{0}^{t_1} \|u_n(t)\|_E^2 \, dt \right)^{1/2}
\]
$$\leq \left( \int_{0}^{t_1} \left( m(n)  b(n) \right) e^{m \alpha (t_1 - t_0)} + m(n) \sqrt{(t_1 - t_0)} e^{m \alpha (t_1 - t_0)} \|u\|_H \, dt \right)^{1/2}.
$$
That is,
\[
\|u_n\| \leq m(n) / (1 - t_0) ( b(t)) e^{m(n)(t - t_0)} + r(t) \|x_0\| e^{m(t_0 - t)} - t_0 \|x_0\|
\]
This shows that \(u_n\) is bounded for a given \(u\) and \(u(o)\).

We now show that \(u_n\) is a Cauchy sequence. Since \(u_n\) is given by
\[
u_n = KNu_n + K u + (1/n)Nu_n,
\]
for some \(p > n\), we have
\[
u_p - u_n = KNu_p - KNu_n + (1/p)Nu_p - (1/n)Nu_n
\]
and
\[
(u_p - u_n, Nu_p - Nu_n) = (KNu_p - KNu_n, Nu_p - Nu_n) + \left((1/p)Nu_p - (1/n)Nu_n, Nu_p - Nu_n\right).
\]

The strong monotonicity of \(-N\) and the monotonicity of \(K\) give the following estimate:
\[
\beta \|u_p - u_n\|^2 \leq \left(\frac{1}{p}\right)\|Nu_p\| + \left(\frac{1}{n}\right)\|Nu_n\| \|Nu_n\| + \|Nu_p\|.
\]

Since \(N\) is a bounded operator and \(u_n\) is bounded, we have \(\|Nu_n\|\) is bounded and let
\[
q = \max(\|Nu(0)\|, \|Nu_p\|), \quad p > n,
\]
for a given \(v\). Then we have
\[
\beta \|u_p - u_n\|^2 \leq \frac{1}{n} 4q^2
\]
and hence for a given \(\epsilon > 0\), we can find an integer \(n_1 = \frac{4q^2}{\beta \epsilon^2}\) such that \(\|u_p - u_n\| < \epsilon\) for \(p, n \geq n_1\). Thus, \(u_n\) is a Cauchy sequence and as \(n \to \infty\), \(u_n\) converges to a solution of (\ref{eq:Raju_3}).

To prove the continuity of \(W\), let \(\{u_k\}\) be a sequence in $X$ such that \(u_k \to u_0\). We have seen that corresponding to each \(u_k\), there exists a unique \(u_k\) such that
\[
x_k = KNx_k + Kv_k.
\]

Since \(\{u_k\}\) is bounded in \(L^2(J, E)\), it converges weakly to \(u(0)\) (say, a subsequence of it which we do not distinguish). Since \(W\) is bounded, \(\{Wu_k\}\) converges to some \(y\) weakly. The compactness of \(K\) implies that \(\{KNu_k\}\) converges strongly and hence \(\{u_k\}\) converges to \(u(0)\) strongly. Now by using the continuity of \(N\) and the uniqueness of the solution, we have
\[
u(0) = KNu(0) + Kv_0 .
\]
Thus, \(W\) is continuous.

Let \(\{u_k\}\) be a bounded sequence in \(L^2(J, E)\). Then in order to prove that \(W\) is compact, we need to show that \(\{Wu_k\}\) converges. Since \(\{u_k\}\) converges weakly and the corresponding \(\{u_k\}\) is bounded, we have that \(\{u_k\}\) converges weakly. Now the compactness of \(K\) implies that \(\{Wu_k\}\) converges strongly and, hence, the lemma.

\end{proof}
It can be easily seen that $K$ is a bounded linear operator and $KN$ is a nonlinear map from $X$ into $E$. Under suitable growth assumptions of $f$ XXXXXXX, it can be shown that $KN$ is continuous and bounded provided $W$ is continuous. We impose the following conditions on $A$, $B$, and $f$:

\begin{itemize}
	\item[(A')] For any given $p \in Z$, $\exists q \in R(B) \subset Z$ such that $L(p) = L(q)$.
	\item[(B')] The solution map $W$ is compact and continuous.
	\item[(C')] The Nemytskii operator $F$ is uniformly bounded in $Z$, that is, $\|Fz\|_Z \leq M$ for all $z \in Z$.
\end{itemize}

\textbf{Remark:} In \cite{naito1987controllability}, approximate controllability was shown by assuming (A'), compactness of $W$, Lipschitz continuity of $F$ in $V$ and uniform boundedness of $F$ in $V$. Here we no longer require Lipschitz continuity of $F$, but instead, we require the continuity of $W$ in $Z$. We also replace the uniform boundedness of $F$ in $V$ by uniform boundedness of $F$ in $Z$, which is a weaker condition. One of the sufficient conditions on $A$ and $F$ for $W$ to be compact and continuous is the following:
\begin{enumerate}
	\item[(a)] The semigroup $S(t)$ is compact for $t > 0$;
	\item[(b)] $\langle -Ax, x \rangle_V \geq \alpha \|x\|_V^2$ for all $x \in V$ and for some constant $\alpha > 0$;
	\item[(c)] $f(t, u(t))$ is a negative monotone, i.e., $\langle f(t, u) - f(t, v), u - v \rangle \leq 0$ for all $u, v \in V$ and $t \in [0, T]$.
\end{enumerate}

\begin{theorem}\label{Theo:th_2}
    Let $A \in A^\alpha(\omega_0, \theta_0)$ with $\theta_0 \in \left(0, \frac{\pi}{2}\right]$ and $\omega_0 \in \mathbb{R}$. Assume that the $\alpha$-order solution operator $\{T_\alpha(t)\}_{t \geq 0}$ and $\alpha$-order resolvent operator $\{S_\alpha(t)\}_{t \geq 0}$ are compact. Suppose also that the assumptions (A')--(C'), are satisfied then FEE (\ref{eq:P}) is approximately controllable on J.
\end{theorem}
\begin{proof}
    Let $R(B) = U \subset X$ and consider the operators $K$ and $KN$ as defined in (\ref{eq:Raju_2}) from $U$ into $E$. From (A') it can be seen that $R(KN) \subset R(K)$ and the operator equation $Kv = \mu$ is approximately solvable \cite[theorem 1]{naito1987controllability}. From conditions (B') and (C') it follows that $KN$ is compact and uniformly bounded. Thus the conditions (A')--(C') of Theorem \ref{theorem:Solvable} are satisfied by $K$ and $KN$ and hence the equation $Ku + KNu = \mu$ is approximately solvable, i.e., $\exists $ a set $D \subset E$ such that $\forall \mu \in D$ one can find $u \in R(B)$ such that
\begin{equation}
	\mu = \int_0^a G(t,s)[u(s) + F(Wu(s))] \, ds. 
\end{equation}

Defining $Wu = y$, we obtain that $\mu = y(a)$.

Now since $u \in R(B)$, for any given $\varepsilon > 0$, $\exists v \in L^2(J,U)$ such that $\|u - Bv\|_X \leq \varepsilon$. Let $W(Bv) = y_v$. Then
\begin{align*}
	\|y(a) - y_v(a)\|_E &\leq \left\| \int_0^a G(T,s)[u(s) - Bv(s)] \, ds \right\|_E\\
	&\quad + \left\| \int_0^a G(T,s)[F(y(s)) - F(y_v(s))] \, ds \right\|_E \\
	&\leq K\|u - Bv\|_Z + K \int_0^a \|F(y(s)) - F(y_v(s))\|_E \, ds \\
	&\leq K\varepsilon^{1/2} + K a^{1/2} \|F(y) - F(y_v)\|_X, 
\end{align*}
that is,
\begin{equation}\label{eq:Last}
    \|y(a) - y_u(a)\|_E \leq K\varepsilon^{1/2} + K a^{1/2} \|F(y) - F(y_u)\|_X, 
\end{equation}
where $\|G(t,s)\| \leq N$ for $0 \leq s \leq t \leq a$.

Consider $W$ and $F$ are continuous in X and $\|u -Bv\|_X \leq \epsilon$, the second term on the  RHS  of (\ref{eq:Last}) can be made arbitrarily small. This implies that the LHS of (\ref{eq:Last}) can be made arbitrarily small by the suitable choice of $u$ which  means $\mathcal{K}_a(0)\subset \overline{\mathcal{K}_a(0)}$. Since $\mathcal{K}_a(0)$ is dense in $E$ it follows that (\ref{eq:P}) is approximately controllable.

\end{proof}

\section{Application}
\label{sec:Appli}
Consider the fractional partial differential equation with discrete nonlocal initial condition of the form
\begin{equation}
\begin{aligned}
&\begin{cases}
{}^C D_t^{\frac{3}{4}} u(x,t) = \frac{\partial^2}{\partial x^2} u(x,t) + \frac{\sin(u(x,t))}{t^2 + 1} + \kappa v(x,t), & x \in [0, \pi], \, t \in J, \\
u(0,t) = u(\pi, t) = 0, & t \in J, \\
u(x,0) = \sum_{k=1}^m c_k u(x, t_k), & x \in [0, \pi],
\end{cases}
\end{aligned}
\label{eq:App_1}
\end{equation}
where ${}^C D_t^{\frac{3}{4}}$ is the Caputo fractional derivative of order $\alpha = \frac{3}{4}$, $\kappa$, $a > 0$ are two constants, $J = [0,a]$, $v \in L^2(J, L^2(0,\pi; \mathbb{R}))$, $c_k \in \mathbb{R}$, $k = 1, 2, \ldots, m$.
Let $E = L^2(0, \pi; \mathbb{R})$ with the norm $\| \cdot \|$. We define the linear operator $A$ in the Hilbert space $E$ by
\begin{equation}\label{eq:App_2}
    A u := \frac{\partial^2}{\partial x^2} u, \quad u \in D(A),
\end{equation}
where
\begin{equation}\label{eq:App_3}
    D(A) := \{ u \in L^2(0, \pi; \mathbb{R}) : u'' \in L^2(0, \pi; \mathbb{R}), \, u(0) = u(\pi) = 0 \}. 
\end{equation}

It is well known from \cite{pruss2012evolutionary} that $A$ generates a compact and analytic $C_0$-semigroup $T(t)$ ($t \geq 0$) in $E$, so that $R(\lambda, A) = (\lambda I - A)^{-1}$ is a compact operator for all $\lambda \in \rho(A)$. Therefore, from the subordination principle \cite[Theorems 3.1 and 3.3]{bazhlekova2001fractional}, , we know that $A$ is the infinitesimal generator of an $\alpha$-order compact and analytic solution operator $\{T_\alpha(t)\}_{t \geq 0}$ of analyticity type $(\omega_0, \theta_0)$, which means that $A \in A^\alpha(\omega_0, \theta_0)$ with $\theta_0 \in \big(0, \frac{\pi}{2} \big]$ and $\omega_0 \in \mathbb{R}$. 

By exponential boundedness of the $C_0$-semigroup $T(t)$ ($t \geq 0$), we know that there exist constants $M \geq 1$ and $\omega \in \mathbb{R}$ such that $\|T(t)\| \leq M e^{\omega t}$ for all $t \geq 0$. Furthermore, by \cite[ Corollary 3.2]{bazhlekova2001fractional}we know that the $\alpha$-order solution operator $\{T_\alpha(t)\}_{t \geq 0}$ and $\alpha$-order resolvent operator $\{S_\alpha(t)\}_{t \geq 0}$ generated by $A$ have the following properties
\begin{equation}\label{eq:App_4}
    \|T_\alpha(t)\| \leq M E_\alpha(\omega t^\alpha), \quad \|S_\alpha(t)\| \leq C E_\alpha(\omega t^\alpha) (1 + t^{1-\alpha}), \quad t \geq 0,
\end{equation}
where $E_\alpha(\cdot)$ is the Mittag-Leffler function. Let
\begin{equation}\label{eq:App_5}
    M_T = M \sup_{t \in J} E_\alpha(\omega t^\alpha), \quad M_S = C \sup_{t \in J} E_\alpha(\omega t^\alpha)(1 + t^{1-\alpha}).
\end{equation}
We get from (\ref{eq:App_5}) that  
\begin{equation}\label{eq:App_6}
    \|T_\alpha(t)\| \leq M_T \quad \text{and} \quad \|S_\alpha(t)\| \leq t^{\alpha-1} M_S \quad \text{for } t \in J.
\end{equation}

Let \( u(t) = u(\cdot, t) \), \( f(t, u(t)) = \frac{\sin(u(\cdot, t))}{t^2 + 1} \). We also define the bounded linear operator \( B : U := E \to E \) by \( Bv(t) := \kappa v(t) \). Then the fractional partial differential equation (\ref{eq:App_1}) can be transformed into the abstract form of FEE (\ref{eq:P}).

\begin{theorem}\label{the_appli}
If \( \sum_{k=1}^m |c_k| < 1 / M_T \), then the fractional partial differential equation (\ref{eq:App_1}) has at least one mild solution \( u \in C([0, \pi] \times J) \) and it is approximately controllable on \( J \).
\end{theorem}

\begin{proof}
By the assumption $\sum_{k=1}^m |c_k| < 1 / M_T$, it is easy to see that the assumption (H1) holds. From the definition of nonlinear term $f$ and bounded linear operator $B$ combined with the above discussion, we can easily verify that the assumptions (H2)', (H3) and (H4) are satisfied with 
\[
\alpha = \frac{3}{4}, \quad \beta = \frac{1}{3}, \quad \gamma = \frac{1}{2}, \quad \varphi(t) = \frac{\sqrt{\pi}}{t^2 + 1}, \quad \text{and} \quad \psi(t) = \kappa v(t).
\]

Therefore, our conclusion follows from Theorem \ref{theorem:Th_1_Mild} and Theorem \ref{Theo:th_2}. This completes the proof of Theorem \ref{the_appli}. 
\end{proof}

\bibliographystyle{unsrtnat} 
\bibliography{references} 

\end{document}